\documentclass{article}
\usepackage[utf8]{inputenc}
\usepackage{subfig}
\usepackage{graphicx} % Required for inserting images
\usepackage{caption,color}   % 24.62
\usepackage{amssymb}
\usepackage{amsthm}
\usepackage{amsmath}
\usepackage{epic}
\usepackage{setspace}
\usepackage{float}
\usepackage{natbib}
\usepackage{multirow}
\usepackage{hyperref}
\usepackage{xcolor}
\usepackage{marginnote}
\usepackage{amsthm}
\usepackage{authblk}

\newtheorem{thm}{Theorem}[section]

\newtheorem{lem}[thm]{Lemma}

\newtheorem{defi}[thm]{Definition}

\title{\textbf{Spectral moments and characteristic polynomials of vertex expansion hypergraphs of graphs}}

\author[1]{\quad Ge Lin\thanks{ linge0717@126.com}}
\author[1]{Changjiang Bu\thanks{ Corresponding author. buchangjiang@hrbeu.edu.cn}}

\affil[1]{School of Mathematical Sciences, Harbin Engineering University, Harbin 150001, PR China }

\date{}

\begin{document}

\maketitle

\begin{abstract}

The $s$-vertex expansion hypergraph $G^{[s]}$ is the $2s$-uniform hypergraph obtained by replacing each vertex of a graph $G$ with $s$ new vertices.
A closed walk in $G$ is called an $s$-multiple closed walk if the number of times it arrives at each vertex of $G$ is divisible by $s$.
In this paper,
we obtain an expression for the spectral moments of $G^{[s]}$ in terms of $s$-multiple closed walks in $G$.
Using these spectral moments,
we give the characteristic polynomial of $G^{[s]}$.

\end{abstract}



\section{Introduction}

In spectral hypergraph theory,
determining the characteristic polynomial of a uniform hypergraph is a fundamental problem.
The characteristic polynomial of a uniform hypergraph is the resultant of a system of homogeneous polynomial equations \cite{qi2005eigenvalues}.
So far,
characteristic polynomials have been obtained for only a few classes of uniform hypergraphs,
including complete $3$-uniform hypergraphs in 2021 \cite{zheng2021complete},
the Fano plane in 2022 \cite{clark2022applications},
and uniform hypertrees \cite{li2024hypertree} and power hypergraphs \cite{chen2024spectra} in 2024.
Other results on the characteristic polynomials of uniform hypergraphs can be found in \cite{bao2020combinatorial,chen2021reduction,cooper2015computing,duan2023characteristic,zheng2020simplifying}.

%A useful tool for computing the characteristic polynomial of a uniform hypergraph is its spectral moments,
%which are the sums of powers of all its eigenvalues.
It is well known that determining the spectral moments of a graph is equivalent to determining its characteristic polynomial \cite{van2003graphs}.
The spectral moments of a graph are equal to the traces of its adjacency matrix \cite{cvetkovic1980spectra}.
%The traces of general tensors are defined by Morozov and Shakirov \cite{morozov2011analogue},
The spectral moments of a uniform hypergraph are also equal to the traces of its adjacency tensor \cite{hu2013determinants}.
In 2015,
Shao et al. \cite{shao2015some} gave some combinatorial formulas for the traces of general tensors,
and established some properties of the characteristic polynomial of $k$-uniform hypergraphs whose spectra are $k$-symmetric.
In 2021,
Clark and Cooper \cite{clark2021harary} generalized the Harary-Sachs theorem to uniform hypergraphs based on tensor traces.
In 2024,
Chen et al. \cite{chen2024spectra} used a trace formula of Shao et al. to derive an expression for the spectral moments of the power hypergraph of a graph in terms of parity-closed walks in the graph.
%Using these spectral moments and all eigenvalues of the power hypergraph given in \cite{CHEN2023205},
%Chen et al. \cite{chen2024spectra} further gave its characteristic polynomial.
The characteristic polynomial of the power hypergraph was further obtained in \cite{chen2024spectra} by using these spectral moments and all its eigenvalues given in \cite{CHEN2023205}.
As an application of spectral moments of power hypergraphs,
Chen et al. \cite{chen2025high} derived some high-order cospectral invariants of trees and used them to distinguish some cospectral trees.
This paper considers the vertex expansion hypergraph of a graph.

%The $s$-vertex expansion hypergraph $G^{[s]}$ is the $2s$-uniform hypergraph obtained by replacing each vertex of a graph $G$ with a set of size $s$ and preserving the adjacency relation.
The $s$-vertex expansion hypergraph of a graph $G$,
denoted by $G^{[s]}$,
is the $2s$-uniform hypergraph obtained by replacing each vertex of $G$ with $s$ new vertices.
Some spectral properties of $G^{[s]}$ have been studied.
In 2015,
Khan and Fan \cite{Fan2015spectral} proved that $G^{[s]}$ has the same spectral radius as $G$.
In 2016,
Khan et al. \cite{khan2016h} determined all $\mathrm{H}$-eigenvalues (i.e., real eigenvalues with real eigenvectors) of $G^{[s]}$ by using the eigenvalues of induced subgraphs of $G$.
In 2019,
Fan et al. \cite{fan2019spectral} characterized the spectral symmetry of $G^{[s]}$.
In \cite{lin2026eigenvalue},
we determined all eigenvalues of $G^{[s]}$ by using the eigenvalues of so-called $2s$-weighted induced subgraphs of $G$.

In this paper,
we obtain an expression for the spectral moments of $G^{[s]}$ in terms of $s$-multiple closed walks in $G$ (see Theorem \ref{dingli1}),
which are closed walks such that the number of times they arrive at each vertex of $G$ is divisible by $s$.
Using these spectral moments and the result in \cite{lin2026eigenvalue},
we determine the multiplicity of each eigenvalue of $G^{[s]}$,
thereby giving its characteristic polynomial (see Theorem \ref{dingli2}).

\section{Preliminaries}

In this section,
we introduce some definitions on tensors and hypergraphs,
as well as some notation for the spectral moments of hypergraphs.
Unless otherwise stated,
this notation is used throughout the paper.
We also provide some auxiliary lemmas required for the proofs in the next section.

A $k$-order $n$-dimensional complex tensor $A=(a_{i_{1}i_{2}\cdots i_{k}})$ is a multi-dimensional array with entries $a_{i_{1}i_{2}\cdots i_{k}}\in\mathbb{C}$ for all $i_{j}\in[n]:=\{1,2,\ldots,n\}$ and $j\in[k]$.
For a vector $\mathbf{x}=(x_{1},x_{2},\ldots,x_{n})^{\top}\in\mathbb{C}^{n}$,
let $\mathbf{x}^{[k-1]}=(x_{1}^{k-1},x_{2}^{k-1},\ldots,x_{n}^{k-1})^{\top}$.
Let $A\mathbf{x}^{k-1}$ be an $n$-dimensional vector with $\sum_{i_{2},\ldots,i_{k}=1}^{n}a_{ii_{2}\cdots i_{k}}x_{i_{2}}\cdots x_{i_{k}}$ as its $i$-th component.
If there exist $\lambda\in\mathbb{C}$ and a non-zero vector $\mathbf{x}$ such that
\begin{align*}
A\mathbf{x}^{k-1}=\lambda\mathbf{x}^{[k-1]},
\end{align*}
then $\lambda$ is called an \emph{eigenvalue} of $A$ and $\mathbf{x}$ is an \emph{eigenvector} of $A$ corresponding to $\lambda$ \cite{lim2005singular,qi2005eigenvalues}.
The \emph{characteristic polynomial} of $A$ is the resultant of the polynomial system $(\lambda\mathbf{x}^{[k-1]}-A\mathbf{x}^{k-1})$ \cite{qi2005eigenvalues}.
The \emph{spectrum} of $A$ is the multiset of the zeros of its characteristic polynomial.
For a positive integer $\ell$,
the spectrum of $A$ is called \emph{$\ell$-symmetric} if it is invariant under a rotation of an angle $\frac{2\pi}{\ell}$ in the complex plane \cite{cooper2012spectra}.

%Let $\ell$ be a positive integer.
%The spectrum $\mathrm{Spec}(A)$ of a tensor $A$ is called $\ell$-\emph{symmetric} if
%\begin{align}\label{eq1}
%\mathrm{Spec}(A)=\mathrm{e}^{\frac{2\pi}{\ell}\mathbf{i}}\mathrm{Spec}(A).
%\end{align}
%The maximum number $\ell$ such that \eqref{eq1} holds is called the \emph{cyclic index} of $A$ \cite{fan2019spectral}.
%Obviously,
%if the cyclic index of $A$ is equal to $\ell$,
%then the spectrum of $A$ is $\ell$-symmetric;
%and for any positive integer $t$ such that $t \mid \ell$,
%the spectrum of $A$ is also $t$-symmetric.

A hypergraph is called $k$-\emph{uniform} if each of its edges contains exactly $k$ vertices.
For a $k$-uniform hypergraph $H=(V(H),E(H))$ with $n$ vertices,
its \emph{adjacency tensor} $A_{H}=(a_{i_{1}i_{2}\cdots i_{k}})$ is a $k$-order $n$-dimensional tensor \cite{cooper2012spectra},
where
\begin{equation*}
a_{i_{1}i_{2}\cdots i_{k}}=\begin{cases}
\frac{1}{(k-1)!},&\text{if $\{i_{1},i_{2},\ldots,i_{k}\}\in E(H)$},\\
0,&\text{otherwise}.
\end{cases}
\end{equation*}
The eigenvalues (resp. spectrum and characteristic polynomial) of $A_{H}$ are called the \emph{eigenvalues} (resp. \emph{spectrum} and \emph{characteristic polynomial}) of $H$.
The sum of $d$-th powers of all eigenvalues of $H$ is called the $d$-th order \emph{spectral moment} of $H$,
denoted by $\mathrm{S}_{d}(H)$.
In order to describe a formula for $\mathrm{S}_{d}(H)$ (see Lemma \ref{yinli1}),
we introduce some related notation.

For a $k$-uniform hypergraph $H=(V(H),E(H))$ with $V(H)=[n]$ and a $k$-tuple $i v_1v_2\cdots v_{k-1}\in [n]^k$,
we interpret it as a rooted hyperedge of $H$ with root $i$ if $\{i, v_1,v_2,\ldots,v_{k-1}\} \in E(H)$.
When no confusion arises,
we use $i v_1v_2\cdots v_{k-1}$ to represent the hyperedge $\{i, v_1,v_2,\ldots,v_{k-1}\}$ rooted at $i$.
Let $\mathcal{F}_{d}=\{ (i_{1}\alpha_{1},\ldots,i_{d}\alpha_{d}): 1 \leq i_{1} \leq \cdots\leq i_{d} \leq n, \alpha_{1},\ldots,\alpha_{d}\in[n]^{k-1}, i_j\alpha_j \in E(H) \mbox{ for all $j \in [d]$} \}$.
Let $f=(i_{1}\alpha_{1},\ldots,i_{d}\alpha_{d})\in\mathcal{F}_{d}$,
where $i_{j}\alpha_{j}\in [n]^{k}$ for all $j \in [d]$.
It implies that $f$ consists of $d$ rooted hyperedges.
Construct a $k$-uniform hypergraph $H_{f}$ such that $V(H_{f})=\bigcup_{j=1}^{d}i_j\alpha_j$ and $E(H_{f})=\bigcup_{j=1}^{d}\{i_j\alpha_j\}$.
Obviously, $H_{f}$ is a subhypergraph of $H$.
For $i_{j}\alpha_{j}=i_{j}v_{1}\cdots v_{k-1}$,
let $\theta(i_{j}\alpha_{j})=\{(i_{j},v_{1}),\ldots,(i_{j},v_{k-1})\}$ be the set of arcs from $i_{j}$ to $v_{1},\ldots,v_{k-1}$.
%where $(v_{1},v_{2})$ denotes an arc from the vertex $v_{1}$ to the vertex $v_{2}$.
Construct a multi-digraph $D_{f}$ such that $V(D_{f})=\bigcup_{j=1}^{d}i_j\alpha_j$ and $E(D_f)=\bigcup_{j=1}^{d}\theta(i_{j}\alpha_{j})$.

Next,
consider a set of representatives of isomorphic connected subhypergraphs
\begin{align*}
\mathcal{H}_{d}=\left\{ \widehat{H}: \mbox{$H_{f}\cong\widehat{H}$ and $D_f$ is Eulerian for some $f\in\mathcal{F}_{d}$} \right\}.
\end{align*}
For $\widehat{H}\in\mathcal{H}_{d}$,
let $\mathcal{F}_{d}(\widehat{H})=\{f:\mbox{$f$ consists of $d$ rooted hyperedges of $\widehat{H}$ and $H_{f}\cong\widehat{H}$}\}$.
Denote the number of subhypergraphs of $H$ which are isomorphic to $\widehat{H}$ by $N_{H}(\widehat{H})$.
Let \[\mathfrak{D}_{d}(\widehat{H})=\left\{D:\mbox{$D_{f}\cong D$ is  Eulerian for some $f\in\mathcal{F}_{d}(\widehat{H})$}\right\}\] be the set of representatives of isomorphic Eulerian multi-digraphs.
Using the trace formula of Shao et al. \cite{shao2015some},
Chen et al. \cite{chen2024spectra} gave the following expression for the spectral moments of $H$ in terms of the numbers of connected subhypergraphs of $H$.

\begin{lem}\citep[(2.3) and Lemma 2.3]{chen2024spectra}\label{yinli1}
The $d$-th order spectral moment of a $k$-uniform hypergraph $H$ with $n$ vertices is
\begin{align}\label{eq2}
\mathrm{S}_{d}(H)=(k-1)^{n-1}\sum_{\widehat{H}\in\mathcal{H}_{d}}c_{d}(\widehat{H})N_{H}(\widehat{H}).
\end{align}
The $d$-th order spectral moment coefficient of the $k$-uniform hypergraph $\widehat{H}$ is
\begin{align}\label{eq3}
c_{d}(\widehat{H})=d(k-1)((k-1)!)^{-d}\sum_{D\in\mathfrak{D}_{d}(\widehat{H})}\frac{\left|\{f:\mbox{$f\in\mathcal{F}_{d}(\widehat{H})$ and $D_{f}\cong D$}\}\right|t(D)}{\prod_{v\in V(D)}\mathrm{deg}_{D}^{+}(v)},
\end{align}
where $\mathrm{deg}_{D}^{+}(v)$ is the out-degree of $v \in V(D)$ and $t(D)$ is the number of spanning trees of $D$.
\end{lem}

The expressions for the number of Eulerian closed walks and spanning trees in an Eulerian multi-digraph are provided in Lemma \ref{yinli+}.
Although these are unrelated to hypergraphs,
they will be used in the next section to reduce the spectral moment coefficients of hypergraphs.

\begin{lem}\cite{Farrell2016multi,Stanley1999enumerative}\label{yinli+}
Let $D=(V(D),E(D))$ be an Eulerian multi-digraph.
Then
\begin{enumerate}

\item[(a)]\citep[Corollary 5.6.6]{Stanley1999enumerative} The number of spanning trees of $D$ is $$t(D)=\frac{1}{|V(D)|}\mu_{1}\cdots\mu_{|V(D)|-1},$$
where $\mu_{1},\cdots,\mu_{|V(D)|-1}$ are the nonzero eigenvalues of the Laplacian matrix of $D$.

\item[(b)]\citep[Theorem 6]{Farrell2016multi} The number of Eulerian closed walks in $D$ is $$\frac{|E(D)|}{b(D)}t(D){\prod_{v\in V(D)}(\mathrm{deg}_{D}^{+}(v)-1)!},$$
where $b(D)$ is the product of the factorials of the multiplicities of the arcs in $E(D)$ and $\mathrm{deg}_{D}^{+}(v)$ is the out-degree of $v \in V(D)$.

\end{enumerate}
\end{lem}

\section{The spectral moments of vertex expansion hypergraphs of graphs}

%In this section,
%we will use Lemma \ref{yinli1},
%the formula for the spectral moments of general hypergraphs,
%to give an expression for the spectral moments of the $s$-vertex expansion hypergraph of a graph in terms of $s$-visit closed walks.

%Fan et al. \cite{fan2019spectral} determined the cyclic index of the adjacency tensor of its $s$-vertex expansion hypergraph $G^{[s]}$ as follows.
%
%\begin{lem}\citep[Lemma 4.6 and Corollary 4.9]{fan2019spectral}\label{yinli2}
%Let $G$ be a graph and $s \ge 2$.
%If $G$ is bipartite, then the cyclic index of $A_{G^{[s]}}$ is $2s$; otherwise, it is $s$.
%\end{lem}
%By Lemma \ref{yinli2},
%we know that the spectrum of $G^{[s]}$ is $s$-symmetric for $s \geq 2$.
%And so the $d$-th order spectral moment $\mathrm{S}_{d}(G^{[s]})=0$ for $s \nmid d$.

For a graph $G=(V,E)$ and $s \ge 2$,
it is known that the spectrum of its $s$-vertex expansion hypergraph $G^{[s]}$ is $s$-symmetric \cite{fan2019spectral}.
So the $d$-th order spectral moment $\mathrm{S}_{d}(G^{[s]})=0$ for $s \nmid d$.
Note that the first $k-1$ order spectral moments of any $k$-uniform hypergraph are all $0$ \cite{cooper2012spectra}.
Therefore, from \eqref{eq2},
we have
\begin{align}\label{eq4}
\mathrm{S}_{d}(G^{[s]})=\begin{cases}
(2s-1)^{|V|s-1}\sum_{\widehat{G}\in\mathcal{G}_{d}}c_{d}(\widehat{G}^{[s]})N_{G}(\widehat{G}), &\text{if $d \geq 2s$ and $s \mid d$,}\\
0, &\text{otherwise,}
\end{cases}
\end{align}
where $\mathcal{G}_d$ is a set of representatives for the isomorphism classes of connected subgraphs $\widehat{G}$ of $G$ satisfying
$\mathfrak{D}_d(\widehat{G}^{[s]}) \neq \emptyset$.
%$\mathcal{G}_{d}=\{ \widehat{G}: \mbox{$\widehat{G}$ is a non-empty subgraph of $G$ and $\mathfrak{D}_{d}(\widehat{G}^{[s]}) \neq \emptyset$} \}$.

In the remainder of this section,
%for $s \geq 2$,
we will further reduce the spectral moment coefficient $c_{d}(\widehat{G}^{[s]})$ in \eqref{eq4} from \eqref{eq3} (see Lemma \ref{yinli8}), thereby obtaining an expression for $\mathrm{S}_{d}(G^{[s]})$ (see Theorem \ref{dingli1}).
In order to accomplish this,
we need to characterize the set of multi-digraphs $\mathfrak{D}_{d}(\widehat{G}^{[s]})$ involved in \eqref{eq3} (see Lemma \ref{yinli4}),
and this set determines the set of subgraphs $\mathcal{G}_{d}$ (see Lemma \ref{yinli5}) and other related parameters in \eqref{eq3} (see Lemma \ref{yinli7}).

Let $\widehat{G} = (\widehat{V},\widehat{E}) \in \mathcal{G}_{d}$.
For all $i \in \widehat{V}$,
let $\mathbf{V}_{i}$ denote a set with $s$ elements corresponding to $i$,
and all of those sets are pairwise disjoint.
The vertex set and hyperedge set of $\widehat{G}^{[s]}$ are $V(\widehat{G}^{[s]})=\bigcup_{i \in \widehat{V}}\mathbf{V}_{i}$ and $E(\widehat{G}^{[s]})=\{\mathbf{V}_{i}\cup\mathbf{V}_{j}: \{i,j\} \in \widehat{E}\}$, respectively.
Denote the multiplicity of the arc $(v,v')$ in a multi-digraph $D$ by $m_{D}(v,v')$.
We provide the following lemma to determine the multiplicities of arcs of multi-digraphs in $\mathfrak{D}_{d}(\widehat{G}^{[s]})$,
which will allow us to find all multi-digraphs in it.

\begin{lem}\label{yinli3}
Let $s \geq 2$ and $d \geq 2s$ such that $s \mid d$.
For $\widehat{G} = (\widehat{V},\widehat{E}) \in \mathcal{G}_{d}$,
let $D \in \mathfrak{D}_{d}(\widehat{G}^{[s]})$ and $f \in \mathcal{F}_{d}(\widehat{G}^{[s]})$ such that $D_{f} \cong D$.
For each $i \in \widehat{V}$ and any two distinct vertices $v,v' \in \mathbf{V}_{i}$,
we have
\begin{enumerate}

\item[(a)] $m_{D_f}(v,v')=m_{D_f}(v',v)$.

\item[(b)] $sm_{D_f}(v,v')=\sum_{\{i,j\} \in \widehat{E}}\sum_{u \in \mathbf{V}_{j}}m_{D_f}(u,v)=\sum_{\{i,j\} \in \widehat{E}}\sum_{u \in \mathbf{V}_{j}}m_{D_f}(u,v')$.

\end{enumerate}
\end{lem}

\begin{proof}
For each $i \in \widehat{V}$ and any two distinct vertices $v,v' \in \mathbf{V}_{i}$,
$v'$ occurs as a non-root vertex in every rooted hyperedge of $f$ rooted at $v$.
So the number of rooted hyperedges of $f$ rooted at $v$ is equal to $m_{D_f}(v,v')$.
From the construction of $D_f$,
we have $$\mathrm{deg}_{D_f}^{+}(v)=(2s-1)m_{D_f}(v,v').$$
Since $v$ occurs as a non-root vertex in rooted hyperedges of $f$ rooted at vertices in $\mathbf{V}_{i}\setminus\{v\}$,
as well as in rooted hyperedges of $f$ rooted at vertices in $\mathbf{V}_{j}$ for $\{i,j\}\in\widehat{E}$,
we have
\begin{align*}
\mathrm{deg}_{D_f}^{-}(v)&=\sum_{u \in \mathbf{V}_{i}\setminus\{v\}}m_{D_f}(u,v)+\sum_{\{i,j\} \in \widehat{E}}\sum_{u \in \mathbf{V}_{j}}m_{D_f}(u,v)\\
&=m_{D_f}(v',v)+\sum_{u \in \mathbf{V}_{i}\setminus\{v,v'\}}m_{D_f}(u,v)+\sum_{\{i,j\} \in \widehat{E}}\sum_{u \in \mathbf{V}_{j}}m_{D_f}(u,v).
\end{align*}
Note that $v$ and $v'$ occur together as non-root vertices in rooted hyperedges of $f$ rooted at vertices in $\mathbf{V}_{i}\setminus\{v,v'\}$ and $\mathbf{V}_{j}$ for $\{i,j\} \in \widehat{E}$.
It implies that $$\mathrm{deg}_{D_f}^{-}(v)=m_{D_f}(v',v)+\sum_{u \in \mathbf{V}_{i}\setminus\{v,v'\}}m_{D_f}(u,v')+\sum_{\{i,j\} \in \widehat{E}}\sum_{u \in \mathbf{V}_{j}}m_{D_f}(u,v').$$
Since $D_f \cong D$ and $D \in \mathfrak{D}_{d}(\widehat{G}^{[s]})$ is Eulerian,
we have $\mathrm{deg}_{D_f}^{+}(v)=\mathrm{deg}_{D_f}^{-}(v)$,
which yields that
\begin{align}\label{eq5}
2sm_{D_f}(v,v')
&=m_{D_f}(v',v)+\sum_{u \in \mathbf{V}_{i}\setminus\{v'\}}m_{D_f}(u,v')+\sum_{\{i,j\} \in \widehat{E}}\sum_{u \in \mathbf{V}_{j}}m_{D_f}(u,v')\\
&=m_{D_f}(v',v)+\mathrm{deg}_{D_f}^{-}(v').\notag
\end{align}
Since $\mathrm{deg}_{D_f}^{-}(v')=\mathrm{deg}_{D_f}^{+}(v')=(2s-1)m_{D_f}(v',v)$,
we get $m_{D_f}(v,v')=m_{D_f}(v',v)$,
that is, (a) holds.
It follows that $\sum_{u \in \mathbf{V}_{i}\setminus\{v'\}}m_{D_f}(u,v')=\sum_{u \in \mathbf{V}_{i}\setminus\{v'\}}m_{D_f}(v',u)\\=(s-1)m_{D_f}(v',v)$.
From \eqref{eq5} and (a),
we get$$sm_{D_f}(v,v')=\sum_{\{i,j\} \in \widehat{E}}\sum_{u \in \mathbf{V}_{j}}m_{D_f}(u,v')=\sum_{\{i,j\} \in \widehat{E}}\sum_{u \in \mathbf{V}_{j}}m_{D_f}(u,v),$$
that is, (b) holds.
\end{proof}

In the following,
we introduce a family of nonnegative integer matrices associated with $\widehat{G}$,
which will be used to determine the set of multi-digraphs $\mathfrak{D}_{d}(\widehat{G}^{[s]})$.

%\begin{defi}\label{dingyi1}
%Let $s \geq 2$ and $d \geq 2s$ such that $s \mid d$.
%For $\widehat{G} = (\widehat{V},\widehat{E}) \in \mathcal{G}_{d}$,
%let $\mathbf{Q}_{d,s}(\widehat{G})$ denote the set of all $|\widehat{V}| \times |\widehat{V}|$ nonnegative integer matrices $Q=(q_{ij})$ satisfying the following three conditions:
%\begin{enumerate}
%\item[(i)] For each $i \in \widehat{V}$, $q_{ii}=0$.
%
%\item[(ii)] For each $\{i,j\} \in \widehat{E}$, $q_{ij}+q_{ji}>0$, and for each $\{i,j\} \notin \widehat{E}$, $q_{ij}=q_{ji}=0$.
%
%\item[(iii)] $Q\mathbf{e}=Q^{\top}\mathbf{e}$, $\mathbf{e}^{\top}Q\mathbf{e}=d$, and $\frac{1}{s}Q\mathbf{e}$ is a positive integer vector,
%where $\mathbf{e}$ is a $|\widehat{V}|$-dimensional all-ones column vector.
%\end{enumerate}
%\end{defi}

\begin{defi}\label{dingyi1}
Let $s \geq 2$ and $d \geq 2s$ such that $s \mid d$.
%For $\widehat{G}=(\widehat{V},\widehat{E})\in\mathcal{G}_{d}$,
Let $Q=(q_{ij})$ be a $|\widehat{V}|\times|\widehat{V}|$ nonnegative integer matrix associated with the graph $\widehat{G}$ satisfying the following three conditions:
\begin{enumerate}
\item[(i)] For each $i\in\widehat{V}$, $q_{ii}=0$.

\item[(ii)] For each $\{i,j\}\in\widehat{E}$, $q_{ij}+q_{ji}>0$, and for each $\{i,j\}\notin\widehat{E}$, $q_{ij}=q_{ji}=0$.

\item[(iii)]
$Q\mathbf{e}=Q^{\top}\mathbf{e}$,
$\mathbf{e}^{\top}Q\mathbf{e}=d$,
and $\frac{1}{s}Q\mathbf{e}$ is a positive integer vector,
where $\mathbf{e}$ is a $|\widehat{V}|$-dimensional all-ones column vector.
\end{enumerate}
We use $\mathbf{Q}_{d,s}(\widehat{G})$ to denote the set of all such matrices $Q$.
\end{defi}

For $Q \in \mathbf{Q}_{d,s}(\widehat{G})$,
let $\mathcal{F}_{d}(\widehat{G}^{[s]},Q)$ denote the set of all $f \in \mathcal{F}_{d}(\widehat{G}^{[s]})$ satisfying the following two conditions:
\begin{enumerate}
%\item $f \in \mathfrak{F}_{d}(\widehat{G}^{[s]})$.

\item[(i)] For each $i \in \widehat{V}$ and each $v \in \mathbf{V}_{i}$,
the number of rooted hyperedges in $f$ rooted at $v$ is $\frac{1}{s}(Q\mathbf{e})_{i}$.

\item[(ii)] For each $\{i,j\} \in \widehat{E}$,
the number of rooted hyperedges $\mathbf{V}_{i} \cup \mathbf{V}_{j}$ in $f$ rooted at vertices in $\mathbf{V}_{i}$ (resp. $\mathbf{V}_{j}$) is $q_{ij}$ (resp. $q_{ji}$).
\end{enumerate}
We characterize the multi-digraphs in $\mathfrak{D}_{d}(\widehat{G}^{[s]})$ as follows.

\begin{lem}\label{yinli4}
Let $s \geq 2$ and $d \geq 2s$ such that $s \mid d$.
For $\widehat{G} = (\widehat{V},\widehat{E}) \in \mathcal{G}_{d}$,
we have
\begin{align*}
\mathfrak{D}_{d}(\widehat{G}^{[s]})=
\bigcup_{Q \in \mathbf{Q}_{d,s}(\widehat{G})}\left\{ D : \mbox{$D_{f} \cong D$ for some $f \in \mathcal{F}_{d}(\widehat{G}^{[s]},Q)$} \right\}.
\end{align*}
\end{lem}

\begin{proof}
We prove this result by establishing the two inclusions.

For any $D \in \mathfrak{D}_{d}(\widehat{G}^{[s]})$,
let $f \in \mathcal{F}_{d}(\widehat{G}^{[s]})$ such that $D_{f} \cong D$.
We will show that there exists $Q \in \mathbf{Q}_{d,s}(\widehat{G})$ such that $f \in \mathcal{F}_{d}(\widehat{G}^{[s]},Q)$.
For each $i \in \widehat{V}$ and any two distinct vertices $v,v' \in \mathbf{V}_{i}$,
since $v'$ occurs as a non-root vertex in every rooted hyperedge of $f$ rooted at $v$,
the number of rooted hyperedges of $f$ rooted at $v$ is equal to $m_{D_f}(v,v')$.
It is known from Lemma \ref{yinli3}(a) that $m_{D_f}(v,v')=m_{D_f}(v',v)$.
Then the numbers of rooted hyperedges in $f$ rooted at each vertex in $\mathbf{V}_{i}$ are all equal,
and we denote it by $\mathrm{q}_{i}$.
Since $D_{f} \cong D$ is Eulerian,
we have $\mathrm{q}_{i} >0$.
Suppose that $q_{ij}$ (resp. $q_{ji}$) is the number of rooted hyperedges $\mathbf{V}_{i} \cup \mathbf{V}_{j}$ in $f$ rooted at vertices in $\mathbf{V}_{i}$ (resp. $\mathbf{V}_{j}$).
Since $H_{f} \cong \widehat{G}^{[s]}$,
we have $q_{ij}+q_{ji}>0$ for each $\{i,j\} \in \widehat{E}$ and $q_{ij}=q_{ji}=0$ for each $\{i,j\} \notin \widehat{E}$.
For each $i \in \widehat{V}$,
note that $\sum_{\{i,j\} \in \widehat{E}}q_{ij}=sm_{D_f}(v,v')=s\mathrm{q}_{i}$ and $\sum_{\{i,j\} \in \widehat{E}}q_{ji}=\sum_{\{i,j\} \in \widehat{E}}\sum_{u \in \mathbf{V}_{j}}m_{D_f}(u,v)=\sum_{\{i,j\} \in \widehat{E}}\sum_{u \in \mathbf{V}_{j}}m_{D_f}(u,v')$.
By Lemma \ref{yinli3}(b),
we have $\sum_{\{i,j\}\in\widehat{E}}q_{ji}=sm_{D_f}(v,v')$,
which yields that $$\sum_{\{i,j\} \in \widehat{E}}q_{ij}=\sum_{\{i,j\}\in\widehat{E}}q_{ji}=s\mathrm{q}_{i}.$$
%Denote $Q=(q_{ij})$ and $\mathbf{q}=(\mathrm{q}_{i})$.
Let $Q$ be the $|\widehat{V}| \times |\widehat{V}|$ matrix with entries $q_{ij}$,
and let $\mathbf{q}$ be the $|\widehat{V}|$-dimensional column vector with components $\mathrm{q}_{i}$.
Then we get $Q\mathbf{e}=Q^{\top}\mathbf{e}=s\mathbf{q}$.
The total multiplicity of arcs in $D_f$ is $$|E(D_f)|=\sum_{i \in \widehat{V}}s\mathrm{q}_{i}(2s-1)=(2s^{2}-s)\sum_{i \in \widehat{V}}\mathrm{q}_{i}.$$
Since $|E(D_f)|=d(2s-1)$,
we get $\sum_{i \in \widehat{V}}\mathrm{q}_{i}=\frac{d}{s}$,
that is, $\mathbf{e}^{\top}\mathbf{q}=\frac{1}{s}\mathbf{e}^{\top}Q\mathbf{e}=\frac{d}{s}$.
Thus,
we have $Q \in \mathbf{Q}_{d,s}(\widehat{G})$ and $f \in \mathcal{F}_{d}(\widehat{G}^{[s]},Q)$.
It implies that $$\mathfrak{D}_{d}(\widehat{G}^{[s]})\subseteq
\bigcup_{Q \in \mathbf{Q}_{d,s}(\widehat{G})}\left\{ D : \mbox{$D_{f} \cong D$ for some $f \in \mathcal{F}_{d}(\widehat{G}^{[s]},Q)$} \right\}.$$

For any $Q \in \mathbf{Q}_{d,s}(\widehat{G})$ and $f \in \mathcal{F}_{d}(\widehat{G}^{[s]},Q)$,
we next show that $D_{f}$ is Eulerian,
which implies that $D_{f} \cong D \in \mathfrak{D}_{d}(\widehat{G}^{[s]})$.
For each $i \in \widehat{V}$ and each $v \in \mathbf{V}_{i}$,
the number of rooted hyperedges in $f$ rooted at $v$ is $\frac{1}{s}(Q\mathbf{e})_{i}$.
Then we have $$\mathrm{deg}_{D_{f}}^{+}(v)=\frac{2s-1}{s}(Q\mathbf{e})_{i}.$$
On the other hand,
$v$ occurs as a non-root vertex in the rooted hyperedges of $f$ rooted at vertices in $\mathbf{V}_{i} \setminus \{v\}$,
with a total number of $\frac{s-1}{s}(Q\mathbf{e})_{i}$,
and also in the rooted hyperedges of $f$ rooted at vertices in $\mathbf{V}_{j}$ for $\{i,j\} \in \widehat{E}$,
with a total number of $q_{ji}$.
Then we have $$\mathrm{deg}_{D_{f}}^{-}(v)=\frac{s-1}{s}(Q\mathbf{e})_{i}+\sum_{\{i,j\} \in \widehat{E}}q_{ji}
=\frac{s-1}{s}(Q\mathbf{e})_{i}+\left(Q^{\top}\mathbf{e}\right)_{i}
=\frac{2s-1}{s}(Q\mathbf{e})_{i}.$$
Thus,
we get $\mathrm{deg}_{D_{f}}^{+}(v)=\mathrm{deg}_{D_{f}}^{-}(v)$,
that is, $D_{f}$ is Eulerian.
It implies that $$\mathfrak{D}_{d}(\widehat{G}^{[s]})\supseteq
\bigcup_{Q \in \mathbf{Q}_{d,s}(\widehat{G})}\left\{ D : \mbox{$D_{f} \cong D$ for some $f \in \mathcal{F}_{d}(\widehat{G}^{[s]},Q)$} \right\}.$$
\end{proof}

%Let $W$ be a closed walk in the graph $G=(V,E)$.
%For each $i \in V$,
%the visit number of $i$ in $W$ is defined as the number of times $W$ arrives at $i$.
%If the visit number of every vertex of $G$ in $W$ is divisible by $s$,
%then $W$ is called an $s$-visit closed walk of $G$.
%If $W$ traverses every edge of $G$ at least once,
%then $W$ is called a covering closed walk of $G$.
A closed walk $w$ in the graph $G$ is called an $s$-multiple closed walk if the number of times $w$ arrives at each vertex of $G$ is divisible by $s$.
If $w$ uses every edge of $G$ at least once,
then it is called covering.
We use $p_{d,s}(G)$ to denote the number of covering $s$-multiple closed walks of length $d$ in $G$.
Let $\mathcal{G}(d,s)$ be a set of representatives for the isomorphism classes of connected subgraphs $\widehat{G}$ of $G$ satisfying $p_{d,s}(\widehat{G})>0$.
Using Lemma \ref{yinli4},
we can give an intuitive description of the set $\mathcal{G}_d$ in \eqref{eq4}.

\begin{lem}\label{yinli5}
Let $s \geq 2$ and $d \geq 2s$ such that $s \mid d$.
Then we have $\mathcal{G}_{d}=\mathcal{G}(d,s)$.
\end{lem}

\begin{proof}
Firstly,
we will show that $\mathcal{G}_{d} \subseteq \mathcal{G}(d,s)$.
For any $\widehat{G} \in \mathcal{G}_{d}$ and $D \in \mathfrak{D}_{d}(\widehat{G}^{[s]})$,
by Lemma \ref{yinli4},
there exists $Q \in \mathbf{Q}_{d,s}(\widehat{G})$ such that $D_f \cong D$ for some $f \in \mathcal{F}_{d}(\widehat{G}^{[s]},Q)$.
%Since $D_f$ is Eulerian and $H_f \cong \widehat{G}^{[s]}$,
%the graph $\widehat{G}$ is connected.
Let $D(Q)$ denote the multi-digraph with adjacency matrix $Q$.
It is known from Definition \ref{dingyi1} that $Q\mathbf{e}=Q^{\top}\mathbf{e}$ and $q_{ij}+q_{ji} > 0$ for each $\{i,j\} \in \widehat{E}$.
So $D(Q)$ is Eulerian.
Moreover, $\mathbf{e}^{\top}Q\mathbf{e}=d$ and $\frac{1}{s}Q\mathbf{e}$ is a positive integer vector,
which implies that the number of arcs of $D(Q)$ is $d$ and the in-degree of each vertex is divisible by $s$.
Then an Eulerian closed walk of $D(Q)$ is a covering $s$-multiple closed walk of length $d$ in $\widehat{G}$.
Thus,
we have $\widehat{G} \in \mathcal{G}(d,s)$,
that is, $\mathcal{G}_{d} \subseteq \mathcal{G}(d,s)$.

Next,
we will show that $\mathcal{G}_{d} \supseteq \mathcal{G}(d,s)$.
For any $\widehat{G} \in \mathcal{G}(d,s)$,
there exists an $s$-multiple closed walk $w$ of length $d$ covering $\widehat{G}$.
We can construct a multi-digraph from the closed walk $w$.
Then it is clearly Eulerian,
and denote its adjacency matrix by $Q_w$.
It is easy to verify that $Q_w \in \mathbf{Q}_{d,s}(\widehat{G})$ and $D_f$ is Eulerian for $f \in \mathcal{F}_{d}(\widehat{G}^{[s]},Q_w)$.
So $D_{f} \cong D \in \mathfrak{D}_{d}(\widehat{G}^{[s]})$,
which implies that $\mathfrak{D}_{d}(\widehat{G}^{[s]})$ is non-empty.
Thus,
we have $\widehat{G} \in \mathcal{G}_{d}$,
that is, $\mathcal{G}_{d} \supseteq \mathcal{G}(d,s)$.
\end{proof}

For $\widehat{G} \in \mathcal{G}(d,s)$ and $Q \in \mathbf{Q}_{d,s}(\widehat{G})$,
we use $D(Q)$ to denote a multi-digraph with adjacency matrix $Q$.
An additional consequence of the proof of Lemma \ref{yinli5} is that $D(Q)$ is an Eulerian multi-digraph and its Eulerian closed walks are covering $s$-multiple closed walks of length $d$ in $\widehat{G}$.
On the other hand,
a covering $s$-multiple closed walk of length $d$ in $\widehat{G}$ is also an Eulerian closed walk of $D(Q)$ for some $Q \in \mathbf{Q}_{d,s}(\widehat{G})$.
By Lemma \ref{yinli+}(b),
we can derive the following result directly.

\begin{lem}\label{yinli6}
Let $s \geq 2$ and $d \geq 2s$ such that $s \mid d$.
For $\widehat{G}=(\widehat{V},\widehat{E}) \in \mathcal{G}(d,s)$,
we have
$$p_{d,s}(\widehat{G})
=\sum_{Q \in \mathbf{Q}_{d,s}(\widehat{G})}\frac{d}{\prod_{\{i,j\}\in\widehat{E}}q_{ij}!q_{ji}!}t(D(Q))\prod_{i\in\widehat{V}}(\mathrm{q}_{i}-1)!,$$
where $t(D(Q))$ is the number of spanning trees of $D(Q)$ and $\mathrm{q}_{i}=(Q\mathbf{e})_{i}$.
\end{lem}

In order to further reduce the spectral moment coefficient $c_{d}(\widehat{G}^{[s]})$ from \eqref{eq3},
some parameters involved in it are also provided as follows,
including the number of spanning trees and the product of out-degrees of vertices.

\begin{lem}\label{yinli7}
Let $s \geq 2$ and $d \geq 2s$ such that $s \mid d$.
Let $\widehat{G}=(\widehat{V},\widehat{E}) \in \mathcal{G}(d,s)$.
For $Q \in \mathbf{Q}_{d,s}(\widehat{G})$ and $f \in \mathcal{F}_{d}(\widehat{G}^{[s]},Q)$,
the number of spanning trees of the multi-digraph $D_f$ is
\begin{align*}
t(D_f)=\frac{1}{s}t(D(Q))\prod_{i\in\widehat{V}}\left(2\mathrm{q}_{i}\right)^{s-1},
\end{align*}
and
\begin{align*}
\prod_{v \in V(D_f)}\mathrm{deg}_{D_f}^{+}(v)=\prod_{i\in\widehat{V}}\left(\frac{(2s-1)\mathrm{q}_{i}}{s}\right)^{s},
\end{align*}
where $\mathrm{q}_{i}=(Q\mathbf{e})_{i}$.
\end{lem}

\begin{proof}
For $Q \in \mathbf{Q}_{d,s}(\widehat{G})$ and $f \in \mathcal{F}_{d}(\widehat{G}^{[s]},Q)$,
we write the Laplacian matrix of the multi-digraph $D_f$ as a block matrix
\begin{align*}
L_{D_f}=
\begin{pmatrix} L_{11} & L_{12} & \cdots & L_{1|\widehat{V}|} \\ L_{21} & L_{22} & \cdots & L_{2|\widehat{V}|} \\ \vdots & \vdots & \ddots & \vdots \\ L_{|\widehat{V}|1} & L_{|\widehat{V}|2} & \cdots & L_{|\widehat{V}||\widehat{V}|}
\end{pmatrix}.
\end{align*}
For each $i \in \widehat{V}$,
the matrix $L_{ii}=L_{D_f}[\mathbf{V}_{i}]$ is an $s \times s$ matrix given by $L_{ii}=2\mathrm{q}_{i}I-\frac{\mathrm{q}_{i}}{s}J$,
where $I$ and $J$ are an identity matrix and an all-ones matrix of size $s \times s$, respectively.
%that is, the diagonal entries are $\frac{(2s-1)\mathrm{q}_{i}}{s}$ and the off-diagonal entries are $-\frac{\mathrm{q}_{i}}{s}$.

For each $i \in \widehat{V}$ and each $v \in \mathbf{V}_{i}$,
let $q_{ij}(v)$ denote the number of rooted hyperedges $\mathbf{V}_{i} \cup \mathbf{V}_{j}$ in $f$ rooted at $v$,
where $j \in \widehat{V}$ is distinct from $i$.
Note that $q_{ij}(v)=0$ if $\{i,j\} \notin \widehat{E}$,
and it is known from the construction of $f \in \mathcal{F}_{d}(\widehat{G}^{[s]},Q)$ that $$\sum_{\{i,j\} \in \widehat{E}}q_{ij}(v)=\frac{\mathrm{q}_{i}}{s}\ \mbox{and}\ \sum_{v \in \mathbf{V}_{i}}q_{ij}(v)=q_{ij}.$$
Let $\mathbf{q}_{ij}$ be an $s$-dimensional column vector with entries $q_{ij}(v)$ for all $v \in \mathbf{V}_{i}$.
The matrix $L_{ij}$ is an $s \times s$ zero matrix for all $\{i,j\} \notin \widehat{E}$,
and $L_{ij}$ is an $s \times s$ matrix given by $L_{ij}=-\mathbf{q}_{ij}\mathbf{\widetilde{e}}^{\top}$ for all $\{i,j\} \in \widehat{E}$,
where $\mathbf{\widetilde{e}}$ is an $s$-dimensional all-ones column vector.

Next,
we will prove that $L_{D_f}$ is orthogonally similar to a block lower triangular matrix,
thereby determining all eigenvalues of $L_{D_f}$ and obtaining the number of spanning trees of $D_f$ by Lemma \ref{yinli+}(a).
Consider the following mutually orthogonal $s$-dimensional column vectors $$(1,1,1,\ldots,1)^{\top}, (1,-1,0,\ldots,0)^{\top}, (1,1,-2,\ldots,0)^{\top}, \ldots , (1,1,1,\ldots,-(s-1))^{\top},$$
and we denote their normalized vectors, in order, by $\mathbf{u}_{1}=\frac{1}{\sqrt{s}}\mathbf{\widetilde{e}},\mathbf{u}_{2},\mathbf{u}_{3},\ldots,\mathbf{u}_{s}$.
Construct an $s \times s$ matrix as follows
$$U=\begin{pmatrix} \mathbf{u}_{1}^{\top} \\ \mathbf{u}_{2}^{\top} \\ \vdots \\ \mathbf{u}_{s}^{\top} \end{pmatrix},$$
and let $\mathbf{U}=\mathrm{diag}(U,U,\ldots,U)$ be a diagonal block matrix with $|\widehat{V}|$ copies of $U$.
We calculate each block submatrix of $\mathbf{U}L_{D_f}\mathbf{U}^{\top}$.
For each $i \in \widehat{V}$,
we have
\begin{align}\label{eq6}
UL_{ii}U^{\top}=U\left(2\mathrm{q}_{i}I-\frac{\mathrm{q}_{i}}{s}J\right)U^{\top}=\mathrm{diag}(\mathrm{q_{i}},2\mathrm{q_{i}},\ldots,2\mathrm{q_{i}}).
\end{align}
For each $\{i,j\} \notin \widehat{E}$,
we have $UL_{ij}U^{\top}=0$.
For each $\{i,j\} \in \widehat{E}$,
since $\mathbf{u}_{1}^{\top} \mathbf{q}_{ij}=\frac{1}{\sqrt{s}}\sum_{v \in \mathbf{V}_{i}}q_{ij}(v)=\frac{q_{ij}}{\sqrt{s}}$
and $\mathbf{\widetilde{e}}^{\top}U^{\top}=(\sqrt{s},0,\ldots,0)$,
we have
\begin{align}\label{eq7}
UL_{ij}U^{\top}=-U\mathbf{q}_{ij}\mathbf{\widetilde{e}}^{\top}U^{\top}=(-\mathfrak{q}_{ij},\mathbf{0},\ldots,\mathbf{0}),
\end{align}
where $\mathfrak{q}_{ij}$ is an $s$-dimensional column vector whose first component is $q_{ij}$.

For all $t \in [|\widehat{V}|s]$,
let $\alpha_{t}$ be the $|\widehat{V}|s$-dimensional column vector whose $t$-th component is $1$ and all other components are $0$.
Let $P$ be a $|\widehat{V}|s \times |\widehat{V}|s$ permutation matrix with the form $$P=\left(\alpha_{1},\alpha_{s+1},\ldots,\alpha_{(|\widehat{V}|-1)s+1},\alpha_{2},\ldots,\alpha_{s},\ldots,\alpha_{(|\widehat{V}|-1)s+2},\ldots,\alpha_{|\widehat{V}|s}\right).$$
From \eqref{eq6} and \eqref{eq7},
we get
\begin{align}\label{eq8}
P^{\top}\mathbf{U}L_{D_f}\mathbf{U}^{\top}P=\begin{pmatrix} L_{D(Q)} & \mathbf{0} \\ M & N \end{pmatrix},
\end{align}
where $L_{D(Q)}$ is the Laplacian matrix of the multi-digraph $D(Q)$ with adjacency matrix $Q$,
and $N$ is a $|\widehat{V}|(s-1) \times |\widehat{V}|(s-1)$ diagonal matrix whose diagonal entries consist of $s-1$ copies of $2\mathrm{q}_{i}$ for all $i \in \widehat{V}$.
Since both $\mathbf{U}$ and $P$ are orthogonal,
$P^{\top}\mathbf{U}$ is also orthogonal,
%Then $P^{\top}\mathcal{U}L_{D}\mathcal{U}^{\top}P$ is orthogonally similar to $L_{D}$,
%and consequently they have the same spectrum.
which implies that the matrix in \eqref{eq8} is orthogonally similar to $L_{D_f}$.
From \eqref{eq8},
the eigenvalues of $L_{D_f}$ consist of the eigenvalues of $L_{D(Q)}$ and $N$.
Note that the eigenvalues $2\mathrm{q}_{i}$ of the diagonal matrix $N$ are nonzero for all $i \in \widehat{V}$,
and the Laplacian matrix $L_{D(Q)}$ has exactly one eigenvalue $0$ because $D(Q)$ is Eulerian.
Let $\mu_{1},\cdots,\mu_{|\widehat{V}|-1},\mu_{|\widehat{V}|}=0$ be the eigenvalues of $L_{D(Q)}$.
By Lemma \ref{yinli+}(a),
we have $$t(D_f)=\frac{1}{|\widehat{V}|s}\mu_{1}\cdots\mu_{|\widehat{V}|-1}\prod_{i \in \widehat{V}}(2\mathrm{q}_{i})^{s-1}.$$
Since $t(D(Q))=\frac{1}{|\widehat{V}|}\mu_{1}\cdots\mu_{|\widehat{V}|-1}$,
we get
$$t(D_f)=\frac{1}{s}t(D(Q))\prod_{i \in \widehat{V}}(2\mathrm{q}_{i})^{s-1}.$$

Recall that the diagonal blocks of $L_{D_f}$ are $L_{ii}=2\mathrm{q}_{i}I-\frac{\mathrm{q}_{i}}{s}J$ for all $i \in \widehat{V}$.
The out-degree of each vertex $v \in \mathbf{V}_{i}$ is equal to the corresponding diagonal entry of $L_{ii}$,
that is, $\mathrm{deg}_{D_f}^{+}(v)=\frac{(2s-1)\mathrm{q}_{i}}{s}$.
Thus,
we get $$\prod_{v \in V(D_f)}\mathrm{deg}_{D_f}^{+}(v)=\prod_{i\in\widehat{V}}\left(\frac{(2s-1)\mathrm{q}_{i}}{s}\right)^{s}.$$
\end{proof}

We are now ready to further reduce the spectral moment coefficient $c_{d}(\widehat{G}^{[s]})$ from \eqref{eq3} using the number of covering $s$-multiple closed walks.

\begin{lem}\label{yinli8}
Let $s \geq 2$ and $d \geq 2s$ such that $s \mid d$.
For $\widehat{G}=(\widehat{V},\widehat{E}) \in \mathcal{G}(d,s)$,
the $d$-th order spectral moment coefficient of $\widehat{G}^{[s]}$ is
$$c_{d}(\widehat{G}^{[s]})=\frac{2^{|\widehat{V}|(s-1)}s^{|\widehat{V}|s-1}}{(2s-1)^{|\widehat{V}|s-1}}p_{d,s}(\widehat{G}).$$
\end{lem}

\begin{proof}
For each $Q \in \mathbf{Q}_{d,s}(\widehat{G})$,
denote $$\mathfrak{D}_{d}(\widehat{G}^{[s]},Q)=\left\{ D : \mbox{$D_{f} \cong D$ for some $f \in \mathcal{F}_{d}(\widehat{G}^{[s]},Q)$} \right\}.$$
It follows from \eqref{eq3} and Lemmas \ref{yinli4} and \ref{yinli5} that
\begin{align*}
c_{d}(\widehat{G}^{[s]})=&d(2s-1)\left((2s-1)!\right)^{-d}\\
&\times\sum_{Q \in \mathbf{Q}_{d,s}(\widehat{G})}\sum_{D \in \mathfrak{D}_{d}(\widehat{G}^{[s]},Q)}\frac{\left|\left\{f:\mbox{$f \in \mathcal{F}_{d}(\widehat{G}^{[s]},Q)$ and $D_{f}\cong D$}\right\}\right|t(D)}{\prod_{v\in V(D)}\mathrm{deg}_{D}^{+}(v)}.
\end{align*}
By Lemma \ref{yinli7},
for $Q \in \mathbf{Q}_{d,s}(\widehat{G})$ and $f \in \mathcal{F}_{d}(\widehat{G}^{[s]},Q)$,
we have $$\frac{t(D_f)}{\prod_{v\in V(D_f)}\mathrm{deg}_{D_f}^{+}(v)}
=\frac{2^{|\widehat{V}|(s-1)}s^{|\widehat{V}|s-1}t(D(Q))}{(2s-1)^{|\widehat{V}|s}\prod_{i \in \widehat{V}}\mathrm{q}_{i}},$$
where $\mathrm{q}_{i}=(Q\mathbf{e})_i$.
It implies that $\frac{t(D)}{\prod_{v\in V(D)}\mathrm{deg}_{D}^{+}(v)}$ depends only on $Q$ for any $D \in \mathfrak{D}_{d}(\widehat{G}^{[s]},Q)$.
Since $$\sum_{D \in \mathfrak{D}_{d}(\widehat{G}^{[s]},Q)}\left|\left\{f:\mbox{$f\in\mathcal{F}_{d}(\widehat{G}^{[s]},Q)$ and $D_{f}\cong D$}\right\}\right|=\left|\mathcal{F}_{d}(\widehat{G}^{[s]},Q)\right|,$$
we get
\begin{align}\label{eq9}
c_{d}(\widehat{G}^{[s]})
=d\left((2s-1)!\right)^{-d}\sum_{Q \in \mathbf{Q}_{d,s}(\widehat{G})}\frac{2^{|\widehat{V}|(s-1)}s^{|\widehat{V}|s-1}t(D(Q))}{(2s-1)^{|\widehat{V}|s-1}\prod_{i \in \widehat{V}}\mathrm{q}_{i}}\left|\mathcal{F}_{d}(\widehat{G}^{[s]},Q)\right|.
\end{align}

For $f=(v_{1}\alpha_{1},\ldots,v_{d}\alpha_{d}) \in \mathcal{F}_{d}(\widehat{G}^{[s]},Q)$,
recall that $v_1 \leq \cdots \leq v_d$ and $v_{j}\alpha_{j}$ represent a rooted hyperedge of $\widehat{G}^{[s]}$ rooted at $v_j$ for all $j \in [d]$.
Moreover,
the number of rooted hyperedges in $f$ rooted at $v$ is $\frac{\mathrm{q}_i}{s}$ for all $v \in \mathbf{V}_{i}$ and each $i \in \widehat{V}$.
Also,
the number of rooted hyperedges $\mathbf{V}_{i} \cup \mathbf{V}_{j}$ in $f$ rooted at vertices in $\mathbf{V}_{i}$ is $q_{ij}$,
and the number of those rooted at vertices in $\mathbf{V}_{j}$ is $q_{ji}$ for each $\{i,j\} \in \widehat{E}$.

In order to count the number of $f$ in $\mathcal{F}_{d}(\widehat{G}^{[s]},Q)$,
note that we can permute the $2s-1$ non-roots of each rooted hyperedge.
We can also permute the rooted hyperedges with the same root.
For each $i \in \widehat{V}$,
a vertex in $\mathbf{V}_{i}$ occurs as a root $\frac{\mathrm{q}_i}{s}$ times,
but permuting the same rooted hyperedges (with rooted hyperedge $\mathbf{V}_{i} \cup \mathbf{V}_{j}$ rooted at vertices in $\mathbf{V}_{i}$ occurring $q_{ij}$ times) leads to same $f$.
It implies that
\begin{align}\label{eq10}
\left|\mathcal{F}_{d}(\widehat{G}^{[s]},Q)\right|
&=\left((2s-1)!\right)^{d}\prod_{i \in \widehat{V}}\frac{(\sum_{v \in \mathbf{V}_{i}}\frac{\mathrm{q}_i}{s})!}{\prod_{\{i,j\} \in \widehat{E}}q_{ij}!}\notag\\
&=\left((2s-1)!\right)^{d}\frac{\prod_{i \in \widehat{V}}\mathrm{q}_{i}!}{\prod_{\{i,j\} \in \widehat{E}}q_{ij}!q_{ji}!}.
\end{align}
Substituting \eqref{eq10} into \eqref{eq9},
we have
\begin{align*}
c_{d}(\widehat{G}^{[s]})
=\frac{2^{|\widehat{V}|(s-1)}s^{|\widehat{V}|s-1}}{(2s-1)^{|\widehat{V}|s-1}}\sum_{Q \in \mathbf{Q}_{d,s}(\widehat{G})}\frac{d}{\prod_{\{i,j\} \in \widehat{E}}q_{ij}!q_{ji}!}t(D(Q))\prod_{i \in \widehat{V}}(\mathrm{q}_{i}-1)!.
\end{align*}
By Lemma \ref{yinli6},
we get $$c_{d}(\widehat{G}^{[s]})=\frac{2^{|\widehat{V}|(s-1)}s^{|\widehat{V}|s-1}}{(2s-1)^{|\widehat{V}|s-1}}p_{d,s}(\widehat{G}).$$
\end{proof}

The spectral moments of the $s$-vertex expansion hypergraph of a graph are given as follows.

\begin{thm}\label{dingli1}
Let $G=(V,E)$ be a graph.
Then the $d$-th order spectral moment of $G^{[s]}$ is
\begin{align}\label{eq11}
\mathrm{S}_{d}(G^{[s]})
=
\scalebox{0.88}{$
\begin{cases}
(2s-1)^{|V|s-1}\sum_{\widehat{G}\in\mathcal{G}(d,s)}\frac{2^{|\widehat{V}|(s-1)}s^{|\widehat{V}|s-1}}{(2s-1)^{|\widehat{V}|s-1}}p_{d,s}(\widehat{G})N_{G}(\widehat{G}), &\text{if $d\geq2s$ and $s \mid d$},\\
0, &\text{otherwise}.
\end{cases}
$}
\end{align}
\end{thm}

\begin{proof}
We first consider the case $s=1$. 
Note that every closed walk of a graph is a $1$-multiple closed walk.
In this case,
$p_{d,1}(\widehat{G})$ denotes the number of covering closed walks of length $d$ in $\widehat{G}$. 
It is known that the $d$-th order spectral moment of $G$ is equal to the number of
closed walks of length $d$ in $G$. 
Therefore,
the result holds for $s=1$. 

For $s\geq2$, 
the result follows directly from \eqref{eq4} and Lemmas \ref{yinli5} and \ref{yinli8}.
\end{proof}

\noindent\textbf{Remark.}
It was shown in \cite{fan2019spectral} that the spectrum of the $s$-vertex expansion hypergraph of a bipartite graph is $2s$-symmetric.
We show that this fact is also reflected in \eqref{eq11}.
Let $G$ be a bipartite graph,
and let $X$ and $Y$ be the bipartition of $V$.
Suppose that $w$ is a covering $s$-multiple closed walk of length $d$ in $G$,
and denote the number of times $w$ arrives at $i \in V$ by $n_{w}(i)$.
Then we have $$\sum_{i \in X}n_{w}(i)+\sum_{i \in Y}n_{w}(i)=d.$$
Since $w$ arrives at vertices in $X$ and $Y$ the same number of times,
i.e., $\sum_{i \in X}n_{w}(i)\\=\sum_{i \in Y}n_{w}(i)$,
we have $2\sum_{i \in X}n_{w}(i)=2\sum_{i \in Y}n_{w}(i)=d$.
It is known that $s \mid n_{w}(i)$ for each $i \in V$ because $w$ is an $s$-multiple closed walk.
Thus,
we get $2s \mid d$.
Consequently,
if there exist covering $s$-multiple closed walks of length $d$ in $G$,
then $2s \mid d$.
Note that any subgraph of $G$ is also bipartite.
It implies that $\mathcal{G}(d,s)=\emptyset$ if $2s \nmid d$ and we have $\mathrm{S}_{d}(G^{[s]})=0$ from \eqref{eq11}.

\section{The characteristic polynomials of vertex expansion hypergraphs of graphs}

In \cite{lin2026eigenvalue},
we determined all eigenvalues of the $s$-vertex expansion hypergraph $G^{[s]}$ by using the eigenvalues of so-called $2s$-weighted graphs on the graph $G$.
In this section,
we will derive an expression for the multiplicities of the eigenvalues of $G^{[s]}$ in terms of these eigenvalues, numbers of $s$-multiple closed walks in $G$ and spectral moments,
thereby giving the characteristic polynomial of $G^{[s]}$.

A $2s$-weighted graph $G_{\pi}$ consists of the underlying graph $G=(V,E)$ and $\pi:V\rightarrow\{\xi\in\mathbb{C}:\xi^{2s}=1\}$.
The matrix $A(G_{\pi})=(a_{ij}^{\pi})$ is the adjacency matrix of $G_{\pi}$,
where
\begin{equation*}
a_{ij}^{\pi}=\begin{cases}
\pi(i)\pi(j),&\mbox{if $\{i,j\} \in E$},\\
0,&\mbox{otherwise}.
\end{cases}
\end{equation*}
The eigenvalues of $A(G_{\pi})$ are called the eigenvalues of $G_{\pi}$.
An induced subgraph of $G_{\pi}$ is called a $2s$-weighted induced subgraph of $G$.
All eigenvalues of the $s$-vertex expansion hypergraph $G^{[s]}$ are given as follows.

\begin{lem}\citep[Theorem 3.1]{lin2026eigenvalue}\label{yinli9}
Let $G$ be a graph and $s\geq2$.
A complex number $\lambda$ is an eigenvalue of $G^{[s]}$ if and only if
$\lambda$ is an eigenvalue of some $2s$-weighted induced subgraph of $G$.
\end{lem}

Let $\Sigma=\{\sigma_{1}, \sigma_{2}, \ldots , \sigma_{\varsigma}\}$ be the set of $s$-th powers of the nonzero eigenvalues of all $2s$-weighted induced subgraphs of the graph $G$.
%Denote $\varsigma=|\Sigma|$,
%and we use $\sigma_{1}, \sigma_{2}, \ldots , \sigma_{\varsigma}$ to denote the elements of $\Sigma$.
By Lemma \ref{yinli9} and the $s$-symmetry of the spectrum of $G^{[s]}$ \cite{fan2019spectral},
its characteristic polynomial $\phi_{G^{[s]}}(\lambda)$ can be written as
\begin{align}\label{eq12}
\phi_{G^{[s]}}(\lambda)=\lambda^{\mu_{0}(s)}\prod_{i=1}^{\varsigma}\left(\lambda^{s}-\sigma_{i}\right)^{\mu_{i}(s)},
\end{align}
where $\mu_{i}(s)$ is the multiplicity of the factor $\lambda^{s}-\sigma_{i}$ for each $i \in [\varsigma]$.
Note that when $s=1$,
every $2$-weighted induced subgraph of $G$ is cospectral with its underlying induced subgraph.
In this case,
if $\sigma_{i} \in \Sigma$ is not an eigenvalue of $G$,
then $\lambda-\sigma_{i}$ is not a factor of $\phi_{G}(\lambda)$,
and thus we have $\mu_{i}(1)=0$.

Let $$\mathrm{M}=\begin{pmatrix} \sigma_{1} & \sigma_{2} & \cdots & \sigma_{\varsigma} \\ \sigma_{1}^{2} & \sigma_{2}^{2} & \cdots & \sigma_{\varsigma}^{2} \\ \vdots & \vdots & \ddots & \vdots \\ \sigma_{1}^{\varsigma} & \sigma_{2}^{\varsigma} & \cdots & \sigma_{\varsigma}^{\varsigma} \end{pmatrix}$$
be a $\varsigma \times \varsigma$ matrix consisting of the elements of $\Sigma$.
Since $\sigma_{1},\sigma_{2},\ldots,\sigma_{\varsigma}$ are distinct and nonzero,
$\mathrm{M}$ is invertible.
Let $\mu(s)=\left(\mu_{1}(s),\mu_{2}(s),\ldots,\mu_{\varsigma}(s)\right)^{\top}$ be the multiplicity vector of the factors,
and let $\mathrm{S}(s)=\left(\mathrm{S}_{s}(G^{[s]}),\mathrm{S}_{2s}(G^{[s]}),\ldots,\mathrm{S}_{\varsigma s}(G^{[s]})\right)^{\top}$ be the spectral moment vector of $G^{[s]}$.
From \eqref{eq12},
the spectral moments of $G^{[s]}$ can be written as $$\mathrm{S}_{\ell s}(G^{[s]})=\sum_{i=1}^{\varsigma}s\mu_{i}(s)\sigma_{i}^{\ell}$$ for all positive integers $\ell$.
Then we get the following system of equations
\begin{align}\label{eq13}
\mathrm{S}(s)=s\mathrm{M}\mu(s).
\end{align}

%Using the expression for the spectral moments of $G^{[s]}$ based on $s$-multiple closed walks (see \eqref{eq11}),
%we next give another expression for $\mathrm{S}(s)$ different from that in \eqref{eq13}. 
Next, 
we use \eqref{eq11} to express $\mathrm{S}(s)$ in terms of the numbers of covering $s$-multiple closed walks in connected subgraphs of $G$.
Let $\mathcal{G}$ be a set of representatives for the isomorphism classes of connected subgraphs of $G$.
Denote $\chi=|\mathcal{G}|$,
and we use $\widehat{G}_{1}=(\widehat{V}_{1},\widehat{E}_{1}), \ldots , \widehat{G}_{\chi}=(\widehat{V}_{\chi},\widehat{E}_{\chi})$ to denote the elements of $\mathcal{G}$.
Let $$\mathrm{P}(s)=\begin{pmatrix} p_{s,s}(\widehat{G}_{1}) & p_{s,s}(\widehat{G}_{2}) & \cdots & p_{s,s}(\widehat{G}_{\chi}) \\ p_{2s,s}(\widehat{G}_{1}) & p_{2s,s}(\widehat{G}_{2}) & \cdots & p_{2s,s}(\widehat{G}_{\chi}) \\ \vdots & \vdots & \ddots & \vdots \\ p_{\varsigma s,s}(\widehat{G}_{1}) & p_{\varsigma s,s}(\widehat{G}_{2}) & \cdots & p_{\varsigma s,s}(\widehat{G}_{\chi}) \end{pmatrix}$$
be a $\varsigma \times \chi$ matrix whose entries are the numbers of covering $s$-multiple closed walks in the corresponding graphs of $\mathcal{G}$.
For each $\ell \in [\varsigma]$ and each $i \in [\chi]$,
if $\widehat{G}_{i}$ does not contain covering $s$-multiple closed walks of length $\ell s$,
then $p_{\ell s,s}(\widehat{G}_{i})=0$,
which implies that $\left(\mathrm{P}(s)\right)_{\ell i}=0$.
In particular,
$\left(\mathrm{P}(s)\right)_{1i}=0$ for all $i \in [\chi]$.
Let $\mathrm{N}=\left(N_{G}(\widehat{G}_{1}),N_{G}(\widehat{G}_{2}),\ldots,N_{G}(\widehat{G}_{\chi})\right)^{\top}$ be the subgraph number vector of $G$.
Let $\mathrm{D}(s)$ be a $\chi \times \chi$ diagonal matrix whose diagonal entries are
$$\left(\mathrm{D}(s)\right)_{ii}=\frac{2^{|\widehat{V}_{i}|(s-1)}s^{|\widehat{V}_{i}|s-1}}{(2s-1)^{|\widehat{V}_{i}|s-1}}$$
for all $i \in [\chi]$.
From \eqref{eq11},
the $\ell s$-th order spectral moment of $G^{[s]}$ is $$\mathrm{S}_{\ell s}(G^{[s]})=(2s-1)^{|V|s-1}\sum_{i=1}^{\chi}\left(\mathrm{P}(s)\right)_{\ell i}\left(\mathrm{D}(s)\right)_{ii}N_{G}(\widehat{G}_{i}).$$
Then we get
\begin{align}\label{eq14}
\mathrm{S}(s)=(2s-1)^{|V|s-1}\mathrm{P}(s)\mathrm{D}(s)\mathrm{N}.
\end{align}

Combining \eqref{eq13} and \eqref{eq14},
we have the following expression for the multiplicity vector of the factors
$$\mu(s)=\frac{(2s-1)^{|V|s-1}}{s}\mathrm{M}^{-1}\mathrm{P}(s)\mathrm{D}(s)\mathrm{N}.$$
It is known that the degree of the characteristic polynomial of $G^{[s]}$ is $|V|s(2s-1)^{|V|s-1}$.
So the multiplicity $\mu_{0}(s)$ of the factor $\lambda$ in \eqref{eq12} is equal to $|V|s(2s-1)^{|V|s-1}$ minus the sum of all components of $s\mu(s)$.
Therefore,
we obtain the characteristic polynomial of $G^{[s]}$ as follows.

\begin{thm}\label{dingli2}
Let $G=(V,E)$ be a graph.
Then the characteristic polynomial of $G^{[s]}$ is
\begin{align*}
\phi_{G^{[s]}}(\lambda)=\lambda^{\mu_{0}(s)}\prod_{i=1}^{\varsigma}\left(\lambda^{s}-\sigma_{i}\right)^{\mu_{i}(s)},
\end{align*}
where $\mu_{i}(s)=\frac{(2s-1)^{|V|s-1}}{s}\left(\mathrm{M}^{-1}\mathrm{P}(s)\mathrm{D}(s)\mathrm{N}\right)_{i}$ for each $i \in [\varsigma]$
and $\mu_{0}(s)=|V|s(2s-1)^{|V|s-1}-s\sum_{i=1}^{\varsigma}\mu_{i}(s)$.
\end{thm}

\begin{figure}[htbp]
\centering
\includegraphics[scale=0.56]{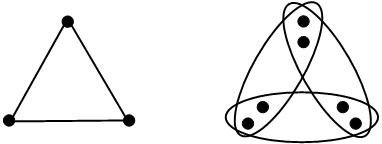}
\caption{The complete graph $K_3$ and its $2$-vertex expansion hypergraph.}
\label{fig1}
\end{figure}

\noindent{\textbf{Example.}}
Let $K_{3}$ be the complete graph with $3$ vertices.
We illustrate Theorem \ref{dingli2} using $K_{3}^{[2]}$ (see Figure \ref{fig1}) as an example.
By Lemma \ref{yinli9},
all eigenvalues of $K_{3}^{[2]}$ are $0$, $\pm 1$, $\pm \mathbf{i}$, $\pm 2$, $\frac{1\pm\sqrt{7}\mathbf{i}}{2}$, and $\frac{-1\pm\sqrt{7}\mathbf{i}}{2}$,
where $\mathbf{i}^{2}=-1$.
So we have $\Sigma=\left\{1, -1, 4, \frac{-3+\sqrt{7}\mathbf{i}}{2}, \frac{-3-\sqrt{7}\mathbf{i}}{2}\right\}$,
which implies that
\[
\mathrm{M} = \begin{pmatrix}
1 & -1 & 4 & \dfrac{-3+\sqrt{7}\mathbf{i}}{2} & \dfrac{-3-\sqrt{7}\mathbf{i}}{2} \\
1 & 1 & 16 & \dfrac{1-3\sqrt{7}\mathbf{i}}{2} & \dfrac{1+3\sqrt{7}\mathbf{i}}{2} \\
1 & -1 & 64 & \dfrac{9+5\sqrt{7}\mathbf{i}}{2} & \dfrac{9-5\sqrt{7}\mathbf{i}}{2} \\
1 & 1 & 256 & \dfrac{-31-3\sqrt{7}\mathbf{i}}{2} & \dfrac{-31+3\sqrt{7}\mathbf{i}}{2} \\
1 & -1 & 1024 & \dfrac{57-11\sqrt{7}\mathbf{i}}{2} & \dfrac{57+11\sqrt{7}\mathbf{i}}{2}
\end{pmatrix}.
\]
Note that $\mathcal{G}=\{K_{1},P_{1},P_{2},K_{3}\}$ and $\mathrm{N}=(3,3,3,1)^{\top}$,
where $K_{1}$ is the single vertex and $P_{\ell}$ is the path of length $\ell$.
By the definition of $\mathrm{P}(2)$ and Lemma \ref{yinli6},
we have $$\mathrm{P}(2)=\begin{pmatrix} 0 & 0 & 0 & 0 \\ 0 & 2 & 0 & 0 \\ 0 & 0 & 0 & 24 \\ 0 & 2 & 12 & 0 \\ 0 & 0 & 0 & 300 \end{pmatrix}.$$
We also have $\mathrm{D}(2)=\mathrm{diag}\left(\frac{4}{3},\frac{32}{27},\frac{256}{243},\frac{256}{243}\right)$.
By Theorem \ref{dingli2},
the characteristic polynomial of $K_{3}^{[2]}$ is
\begin{align*}
\phi_{K_{3}^{[2]}}(\lambda)=&\lambda^{498}\left(\lambda^{2}-1\right)^{208}\left(\lambda^{2}+1\right)^{48}\left(\lambda^{2}-4\right)^{32}\\
&\times\left(\lambda^{2}-\frac{-3+\sqrt{7}\mathbf{i}}{2}\right)^{96}\left(\lambda^{2}-\frac{-3-\sqrt{7}\mathbf{i}}{2}\right)^{96}.
\end{align*}

\section*{Acknowledgments}

The research of the second author is partially supported by the National Natural Science Foundation of China (No. 12371344, 12671402).

%\noindent\unnumbered{\textbf{Data Availability}}
%This manuscript has no associated data.

%\section*{Declarations}
%
%\textbf{Conflict of interest} The authors declare no conflict of interest to the content of this article.

\bibliographystyle{plain}
\bibliography{cpvebib}

\end{document}